\documentclass[11pt,a4paper]{amsart}
\usepackage{amssymb}
\usepackage{graphicx}
\usepackage[colorlinks,
linkcolor=blue,
anchorcolor=blue,
citecolor=blue]{hyperref}

\theoremstyle{plain}
\newtheorem{Main}{Theorem}

\newtheorem{Thm}{Theorem}[section]
\newtheorem{Prop}[Thm]{Proposition}
\newtheorem{Lem}[Thm]{Lemma}

\theoremstyle{remark}
\newtheorem{Rem}[Thm]{Remark}

\def\max{\operatorname{max}}
\def\min{\operatorname{min}}

\def\supp{\operatorname{supp}}
\def\W{\operatorname{W}}

\begin{document}
	
	\title[]
	{the dimension of the non-dense orbit set for expanding maps}
	\author{Congcong Qu}
	\address{Congcong Qu, College of Big Data and software Engineering, Zhejiang Wanli University, Ningbo, 315107, Zhejiang, P.R.China}
	\email{congcongqu@foxmail.com}
	\iffalse
	\author{Yongluo Cao \textsuperscript{*}}
	\address{Yongluo Cao, Department of Mathematics, Soochow University,
		Suzhou 215006, Jiangsu, P.R. China}
	\address{and Center for Dynamical Systems and Differential Equations, Soochow University, Suzhou 215006, Jiangsu, P.R. China}
	\email{ylcao@suda.edu.cn}
	
	\date{\today}

	\thanks{
		* Y. Cao is the corresponding author.
	}
	\fi
	

	\begin{abstract}
		In this paper, we give a new proof for  the Hausdorff dimension of the non-dense orbit set for expanding maps. This proof is based on the sharp lower bound of the Hausdorff dimension of  repellers given by Cao, Pesin and Zhao in~\cite{caopesinzhao2019}. Besides, we prove that the Carath$\acute{\text{e}}$odory singular dimension of non-dense orbit set equals that of the repeller.
	\end{abstract}
	
	\maketitle
	
	%\tableofcontents
	
	\section{Introduction}
	Suppose ~$M$ is a compact~Riemannian manifold and~$f:M\rightarrow M$ is a continuous map. Given an~$f-$invariant set~$\Lambda\subset M$ and a subset~$A\subset \Lambda$, the ~$A-$exceptional set in~$\Lambda$ with respect to $f|_{\Lambda}$ is defined as~$E^{+}_{f|_{\Lambda}}(A)=\{x\in \Lambda:\overline{\mathcal{O}_{f}(x)}\cap A={\varnothing}\}$, where $\mathcal{O}_{f}(x)$ denotes the forward orbit of $x$ under the action of $f$. Denote by $ND(f)$ the set of points with non-dense forward orbit under $f$. By definition, $E^+_{f|_M}(y)\subset ND(f)$ for any $y\in M$. If $f$ preserves an ergodic measure of full support, then $ND(f)$ has measure $0$. It's natural to consider the dimension of the non-dense orbit set.
	
	In this paper, we consider the  dimension of the non-dense orbit set for $C^{1+\gamma}$ expanding maps. Precisely, we prove the following results. Here we use $E^{+}_{f}(y)$ shortly for $E^{+}_{f|_{M}}(y)$. We use $dim_C Z$ to denote Carath$\acute{\text{e}}$odory singular dimension of the set $Z$, see subsection \ref{C dimension} for definition.
	\begin{Main}\label{expanding}
		Let $M$ be a $d-$dimensional Riemannian manifold and $f: M\rightarrow M$ be a $C^{1+\gamma}$ expanding map. Suppose that $\mu$ is an ergodic measure for $f$ which is absolutely continuous to the Lebesgue measure. Then for any $y\in M$, we have $$dim_{H} E^{+}_{f}(y)=dim_{H}\mu=d.$$
	\end{Main}
	\begin{Main}\label{expanding2}
		Let $M$ be a $d-$dimensional Riemannian manifold and $f: M\rightarrow M$ be a $C^{1+\gamma}$ map admitting a repeller $\Lambda$. Then for any $y\in \Lambda$, we have $$dim_{C} E^{+}_{f|_{\Lambda}}(y)=dim_{C} \Lambda.$$
	\end{Main}
	\iffalse
	\begin{Main}\label{hyperbolic1}
		Let $M$ be a $d-$dimensional Riemannian manifold and $f: M\rightarrow M$ be a transitive $C^{1+\alpha}$ hyperbolic diffeomorphism. Suppose that $\mu$ is the unique SRB measure for $f$. Then for any $x, y\in M$, we have $$dim_{H} E^{+}_{f}(y)\cap W^{u}(x)=dim_{H}\mu_{x}^{u}=dim W^{u}(x),$$
		where $\mu_{x}^{u}$ is the conditional measure on the unstable leaf.
	\end{Main}
	\begin{Main}\label{hyperbolic2}
		Let $M$ be a $d-$dimensional Riemannian manifold and $f: M\rightarrow M$ be a $C^{1+\alpha}$ diffeomorphism. Suppose $\Lambda\subset M$ is a locally maximal hyperbolic invariant set. Then for any $x,y\in\Lambda$, we have $$dim_{C} E^{+}_{f|_{\Lambda}}(y)\cap \W^{u}(x)=dim_{C} \Lambda\cap W^{u}(x).$$
	\end{Main}
\fi
	For the history of the research on the dimension of the exceptional set, one can refer to ~\cite{CG1, CG2}. Now we state the differences of our results with the existing ones. The first result for the dimensions of the exceptional sets for expanding maps and hyperbolic diffeomorphsims was obtained by  Urb$\acute{\text{a}}$nski in \cite{Urbanski}. He showed the results of our Main Theorem ~\ref{expanding}  and for Main Theorem ~\ref{expanding2} in conformal case. Notice that for conformal maps, the Hausdorff dimension and the Carath$\acute{\text{e}}$odory singular dimension coincide. He showed the results by a good lower bound of the Hausdorff dimension of a suitable invariant measure supported on the manifold, i.e. volume measure for expanding maps and the Gibbs measure for Anosov diffeomorphisms respectively. And Tseng claimed in \cite{Tseng1} there is a minor gap in Urb$\acute{\text{a}}$nski's proof for the lower bound of the Hausdorff dimension of the measure, and gave a modification of that by playing the Schmidt game, which is developed by Schmidt in \cite{Schmidt}, for expanding endomorphisms on the circle in~\cite{Tseng1} and for certain Anosov diffeomorphisms on the two dimensional torus in ~\cite{Tseng2}. Subsequent work is due to Wu in ~\cite{Wu1, Wu2, Wu3}. His results are based on the tool of Schmidt game and the modified Schmidt game which was introduced by Kleinbock and Weiss in \cite{KW}. Wu extended Tseng's results to the general expanding endomorphisms and partially hyperbolic diffeomorphisms admitting unstable directions.

	Other interesting results are due to Campos and Gelfert in ~\cite{CG1,CG2}. In ~\cite{CG1} and ~\cite{CG2}, for conformal $C^{1+\gamma}$ endomorphisms and diffeomorphisms admitting a locally maximal invariant set, Campos and Gelfert proved that for a set with either small entropy or dimension, the exceptional set of such a set is large in the sense of the Hausdorff dimension and entropy. Their work was heavily based on Dolgopyat's result on entropy in \cite{Do} on the symbolic spaces. They constructed repellers and hyperbolic sets inside the locally maximal invariant set respectively to approximate the entropy and the Lyapunov exponents. The semi-conjugacy and conjugacy respectively of the repellers and hyperbolic sets to the symbolic systems yields the result for the entropy. And the relation of the Hausdorff dimension, entropy and the Lyapunov exponents by Young~\cite{Young} gives the results of dimension. The first author and Juan Wang extended these results to average conformal case in~\cite{QW}.
	
	Our proof of Theorem~\ref{expanding} is based on the lower bound of the Hausdorff dimension of a repeller obtained by Cao, Pesin and Zhao in \cite{caopesinzhao2019}. We briefly state our proof here. For an ergodic measure $\mu$ for $f$ which is absolutely continuous with the Lebesgue measure and a given point $y\in M$, we can construct a sequence of repellers $\{\Lambda_n\}_{n\in\mathbb{N}}$  to approximate the measure $\mu$ and  $\{\Lambda_n\}_{n\in\mathbb{N}}$ can be chosen to avoid the point $y$ and hence $\Lambda_n \subset E^{+}_{f}(y)$. Since $\mu$ is absolutely continuous with respect to the Lesbegue measure, the Pesin's formula holds. We can use this to deduce that the Hausdorff dimension of $\{\Lambda_n\}_{n\in\mathbb{N}}$ are close to the dimension of the manifold via the lower bound of the Hausdorff dimension of the repellers obtained in~\cite{caopesinzhao2019}. By the monotonicity of the Hausdorff dimension, we can show that $dim_H E^{+}_{f}(y)$, $dim_H \mu$ and the dimension of the manifold coincide.
	
	The results in Theorem~\ref{expanding2} is brand new. Cao, Pesin and Zhao \cite{caopesinzhao2019} obtained that the Carath$\acute{\text{e}}$odory singular dimension of the repeller is exactly the unique root of the Bowen's equation involving the sub-additive topological pressure of the sub-additive singular-valued potential. Thus for a given repeller $\Lambda$ and $y\in \Lambda$, we can find a sequence of repellers $\Lambda_n\subset \Lambda$ and $\Lambda_n\cap \{y\}=\varnothing$, and then use  the Carath$\acute{\text{e}}$odory singular dimension of $\Lambda_n$ to approximate that of $\Lambda$.
	\iffalse
	The idea of the proof of Main Theorem ~\ref{hyperbolic1} and ~\ref{hyperbolic2} are similar to the proof of the Main Theorem~\ref{expanding} and ~\ref{expanding2} respectively.
	\fi
	
	This paper is organized as follows. In section \ref{Preliminary}, we recall some basic definitions and preliminary results. In section \ref{proof}, we give the details of the proof of the main result.
	\section{Some definitions and preliminary results}\label{Preliminary}
	In this section, we recall some notions used in our proof and state some preliminary results.

	\subsection{Repellers for expanding maps.} Let $M$ be a $d-$dimensional compact $C^{\infty}$ Riemannian manifold  and $f:M\rightarrow M$ be a $C^{1}$ map. Suppose $\Lambda$ is a compact $f-$invariant subset of $M$. We say $\Lambda$ is {\it a repeller for $f$ and $f$ is expanding} if there exists an open neighborhood $U$ of $\Lambda$ such that $\Lambda=\{x\in U: f^{n}(x)\in U, \text{for~all}~n\in\mathbb{N}\}$ and there is $\lambda>0$ such that for all $x\in \Lambda$ and $v\in T_{x}M$, we have $\|D_{x}f (v)\|\geq \lambda \|v\|$, where $\|\cdot\|$ is the norm induced by the Riemannian metric on $M$.
	\subsection{Metric entropy, topological entropy and the classical topological pressure}
	For the definitions of the metric entropy, topological entropy and the classical topological pressure, one can refer to ~\cite{Pesin} or~\cite{WALTERS}. Here we recall the variational principle of the topological entropy.
	
	Consider a continuous transformation $f: X\to X$ of a compact metric space $X$ equipped with a metric $d$. Suppose that $\Lambda\subset X$ is an $f-$invariant compact set. For an invariant measure $\mu$ for $f$ which is supported on $\Lambda$, we denote by $h_{\mu}(f|_{\Lambda})$ the metric entropy and by $h_{\text{top}}(f|_{\Lambda})$ the topological entropy on $\Lambda$. Classical result shows the following variational principle holds. For the proof, one can refer to ~\cite{WALTERS}.
	\begin{Thm}[Variational principle]
		\begin{align*}
			h_{\text{top}}(f|_{\Lambda})=\sup\{h_{\mu}(f|_{\Lambda})|\supp(\mu)\subset\Lambda,~\mu~\text{is~ergodic}\}.
		\end{align*}
	\end{Thm}
	The classical result shows that for an expansive map, the map $\mu\mapsto h_{\mu}(f|_{\Lambda})$ is upper semi-continuous. For this, one can refer to ~\cite{WALTERS}. Notice that expanding maps  are expansive. Hence one can find an ergodic measure to achieve the supreme.
	\subsection{Sub-additive topological pressure and super-additive topological pressure.}
	
	Consider a continuous transformation $f: X\to X$ of a compact metric space $X$ equipped with a metric $d$. Denote by $\mathcal{M}(f)$ the set of all $f$-invariant Borel probability measures on $X$. We recall that a sequence of continuous functions (potentials) $\Phi=\{\varphi_n\}_{n\geq1}$ is called {\it sub-additive} if
	$$\varphi_{m+n}(x)\leq \varphi_n(x)+ \varphi_m\circ f^n(x),\ \  \forall m,n\in\mathbb{N} , \forall x \in X.$$

	Now we state the definition of sub-additive topological pressure which was introduced in \cite{cfh}.  For $x,y\in X$ and $n\geq 0$, define the metric $d_n$ on $X$ by
	$$d_n(x, y) = \max\{d(f^i(x),f^i(y))|0\leq i<n\}.$$
	Given $\varepsilon>0$ and $n\geq 0$, denote by $B_n(x,\varepsilon)= \{y\in X |d_n(x,y)<\varepsilon\}$ the Bowen's ball centered at $x$ of radius $\varepsilon$ and length $n$. We say a subset $E\subset X$ is {\it $(n,\varepsilon)$-separated} if $d_n(x,y)\geq\varepsilon$ for any two distinct points $x,y\in E$.
	
	Given a sub-additive sequence of continuous potentials $\Phi=\{\varphi_n\}_{n\geq 1}$, let
	$$P_n(\Phi,\varepsilon)= \sup\left\{\sum_{x\in E}e^{\varphi_n(x)}|E \mbox{ is an }(n,\varepsilon) \mbox{-separated set of }X\right\}.$$
	The quantity
	$$P(f,\Phi) =\lim_{\varepsilon\to 0}\limsup_{n\to\infty}\frac{1}{n}\log P_n(\Phi,\varepsilon)$$
	is called {\it the sub-additive topological pressure of $\Phi$}. One can show that it satisfies the following variational principle (see Theorem 1.1 of \cite{cfh}):
	\begin{equation}\label{vp}P(f,\Phi) =\sup_{\mu\in \mathcal{M}(f)}\left\{h_\mu(f)+\mathcal{F}_*(\Phi,\mu)|\mathcal{F}_*(\Phi,\mu)\neq\infty\right\},\end{equation}
	where $h_\mu(f)$ is the metric entropy of $f$ with respect to $\mu$ and
	$$\mathcal{F}_*(\Phi,\mu)=\lim_{n\to\infty}\frac{1}{n}\int\varphi_nd\mu.$$
	Existence of the above limit can be shown by the standard sub-additive argument.
	
	Next we recall the definition of the super-additive topological pressure introduced in \cite{caopesinzhao2019}.
	
	Given a sequence of super-additive continuous potentials $\Psi=\{\psi_{n}\}_{n\geq 1}$, that is a sequence of continuous functions satisfying
	\begin{align*}
		\psi_{m+n}(x)\geq \psi_{n}(f^{m}(x))+\psi_{m}(x)
	\end{align*}
	for any $n,m\in\mathbb{N}$ and any $x\in X$. The {\it super-additive topological pressure of $\Psi$} is defined as
	\begin{align*}
		P_{\text{var}}(f,\Psi)\triangleq \sup \{h_{\mu}(f)+\mathcal{F}_{\ast}(\Psi,\mu):\mu\in \mathcal{M}(f)\},
	\end{align*}
	where $\mathcal{F}_{\ast}(\Psi,\mu)=\lim_{n\rightarrow \infty}\frac{1}{n}\int \psi_{n}d\mu=\sup \frac{1}{n}\int \psi_{n}d\mu$. The second equality is due to the standard sub-additive argument.

	\subsection{Singular-valued potentials.}\label{sub}
	
	Suppose $M$ is a $d$-dimensional compact Riemannian manifold and $f:M\rightarrow M$ is a $C^{1+\gamma}$ expanding map with a repeller $\Lambda$ such that $f|_{\Lambda}$ is topological transitive. Given $x\in M$ and $n\geq 1$, we consider the differentiable operator $D_{x}f^{n}: T_{x}M\rightarrow T_{f^{n}(x)}M$ and denote its singular values in the decreasing order by
	$$\alpha_{1}(x,f^{n})\geq\alpha_{2}(x,f^{n})\geq\cdots\geq \alpha_{d}(x,f^{n}).$$
	For $s\in [0,d]$, set
	$$\varphi^{s}(x,f^{n})\triangleq\sum^{d}_{i=d-[s]+1} \log \alpha_{i}(x,f^{n})+(s-[s])\log\alpha_{d-[s]}(x,f^{n})$$
	and
	\begin{displaymath}
		\psi^{s}(x,f^{n})=\sum_{i=1}^{[s]}\log \alpha_{i}(x,f^{n})+(s-[s])\log \alpha_{[s]+1}(x,f^{n}),
	\end{displaymath}
	where $[s]$ is the largest integer not exceeding $s$. Since $f$ is a $C^{1+\gamma}$ map, the functions $x\mapsto \alpha_{i}(x,f^{n})$, $x\mapsto \varphi^{s}(x,f^{n})$ and $x\mapsto \psi^{s}(x,f^{n})$ are continuous.  It is easy to check that for all $m,n \in \mathbb{N}$,
	$$\varphi^{s}(x,f^{m+n})\geq \varphi^{s}(x,f^{n})+\varphi^{s}(f^n(x),f^{m}),$$
	\begin{align*}
		\psi^{s}(x,f^{m+n})\leq \psi^{s}(x,f^{n})+\psi^{s}(f^{n}(x),f^{m}).
	\end{align*}
	Thus $\Phi_{f}(s)\triangleq \{-\varphi^{s}(\cdot,f^{n})\}_{n\geq 1}$ are sub-additive and $\Psi_{f}(s)=\{-\psi^{s}(\cdot,f^{n})\}_{n\geq 0}$ are super-additive. We call them the {\it sub-additive singular-valued potentials} and the {\it super-additive singular-valued potentials }respectively.
	\iffalse
	The definitions of the singular valued potentials for hyperbolic diffeomorphisms are similar.
	
	Suppose $M$ is a $d$-dimensional compact Riemannian manifold and $f:M\rightarrow M$ is a $C^{1+\alpha}$ diffeomorphism admitting a hyperbolic set $\Lambda$ such that $f|_{\Lambda}$ is topological transitive. Given $x\in M$ and $n\geq 1$, we consider the differentiable operator $D_{x}f^{n}: T_{x}M\rightarrow T_{f^{n}(x)}M$ and denote its singular values in the decreasing order by
	$$\alpha_{1}(x,f^{n})\geq\alpha_{2}(x,f^{n})\geq\cdots\geq \alpha_{u}(x,f^{n})>1>\alpha_{u+1}(x,f^{n})\geq\cdots\geq \alpha_{d}(x,f^{n}),$$ where $u$ is the dimension of the unstable manifold.
	For $s\in [0,u]$, set
	$$\tilde{\varphi}^{s}(x,f^{n})\triangleq\sum^{u}_{i=u-[s]+1} \log \alpha_{i}(x,f^{n})+(s-[s])\log\alpha_{u-[s]}(x,f^{n}),$$
	and
	\begin{align*}
		\tilde{\psi}^{s}(x,f^{n})=\sum_{i=1}^{[s]}\log \alpha_{i}(x,f^{n})+(s-[s])\log \alpha_{[s]+1}(x,f^{n}).
	\end{align*}
	$\tilde{\Phi}_{f}(s)\triangleq \{-\tilde{\varphi}^{s}(\cdot,f^{n})\}_{n\geq 1}$ are sub-additive and $\tilde{\Psi}_{f}(s)=\{-\psi^{s}(\cdot,f^{n})\}_{n\geq 0}$ are super-additive. We call them the sub-additive singular valued potentials and the super-additive singular valued potentials respectively.
	\fi

	\subsection{Carath$\acute{\text{e}}$odory singular dimension.}\label{C dimension} We recall the definition of the Carath$\acute{\text{e}}$odory singular dimension defined in~\cite{caopesinzhao2019}, which is based on the Carath$\acute{\text{e}}$odory construction described in~\cite{Pesin}.
	
	Let $\Lambda$ be a repeller for an expanding map $f$ and $\Phi=\{-\varphi^{\alpha}(\cdot,f^{n})\}_{n\geq 1}$ be the sub-additive singular-valued potential defined in  subsection \ref{sub}. For every $Z\subset \Lambda$, define
	\begin{align*}
		m(Z,\alpha,r)\triangleq \lim_{N\rightarrow \infty} \inf\{\sum_{i}\exp(\sup_{y\in B_{n_{i}}(x_{i},r)} -\varphi^{\alpha}(y,f^{n_{i}}))\},
	\end{align*}
	where the infimum is taken over all the collections $\{B_{n_{i}}(x_{i},r)\}$ of Bowen's balls with $x_{i}\in \Lambda$, $n_{i}\geq N$ those cover $Z$. One can show that there is a jump-up value
	\begin{align*}
		\dim_{C,r} Z\triangleq \inf\{\alpha:m(Z,\alpha,r)=0\}=\sup\{\alpha:m(Z,\alpha,r)=+\infty\},
	\end{align*}
	which is called {\it the Carath$\acute{\text{e}}$odory singular dimension of $Z$}. For the motivation of this definition, one can refer to \cite{caopesinzhao2019}.
	\iffalse
	Similarly, one can define the Carath$\acute{\text{e}}$odory singular dimension of any subset of a hyperbolic set. For this, we only need to replace the sub-additive singular valued potential $\Phi$ by $\tilde{\Phi}$, which is defined in  subsection \ref{sub}.
	\fi
	
	\subsection{Hausdorff dimension}
	Now we recall the definition of the Hausdorff dimension, one can refer to \cite{Pesin}.
	
	Let $X$ be a compact metric space equipped with a metric $d$. Given  a subset $Z$ of $X$, for $s\geq 0$ and $\delta>0$, define
	\[
	\mathcal{H}_{\delta}^{s}(Z):=\inf \left\{\sum_{i}|U_i|^s: \
	Z\subset \bigcup_{i}U_i,~|U_i|\leq \delta\right\}
	\]
	where $|\cdot|$ denotes the diameter of a set. The quantity
	$\mathcal{H}^{s}(Z):=\lim\limits_{\delta\rightarrow 0}\mathcal{H}_{\delta}^{s}(Z)$ is called the {\em $s-$dimensional Hausdorff measure} of $Z$. Define the {\em Hausdorff dimension} of $Z$, denoted by $\dim_H  Z$, as follows:
	\[
	\dim_H  Z =\inf \{s:\ \mathcal{H}^{s}(Z)=0\}=\sup \{s: \mathcal{H}^{s}(Z)=\infty\}.
	\]
	One can check that the Hausdorff dimension satisfies the monotonicity, that is, for $Y_{1}\subset Y_{2}\subset X$, we have ~$\dim_{H}Y_{1}\leq \dim_{H}Y_{2}$.

	Given a probability Borel measure $\mu$ on $X$, the {\it Hausdorff dimension of the measure $\mu$}  is defined as $$\dim_{H}\mu=\inf\{\dim_{H}Y:Y\subset X ,  \mu(Y)=1\}.$$ This quantity was introduced by Young in \cite{Young}.
	\subsection{Lower bound for the Hausdorff dimension of a non-conformal repeller.}
	Recently, Cao, Pesin and Zhao~\cite{caopesinzhao2019} gave a lower bound for the Hausdorff dimension of a non-conformal repeller, which is given by the root of the Bowen's equation involving the super-additive topological pressure of the super-additive singular valued potential. Take $\Psi_{f}(s)=\{-\psi^{s}(\cdot,f^{n})\}_{n\geq 0}$ as the super-additive singular-valued potentials and  $P_{\text{sup}}(s)=P_{\text{var}}(f|_{\Lambda},\Psi_{f}(s))$ as the super-additive topological pressure of the super-additive singular-valued potential.
	\begin{Thm}\cite{caopesinzhao2019}\label{lower bound}
		Let $\Lambda$ be a repeller for a $C^{1+\gamma}$ expanding map $f:M\rightarrow M$. Then
		\begin{align*}
			dim_{H}\Lambda\geq s^{\ast},
		\end{align*}
		where $s^{\ast}$ is the unique root of the equation $P_{\text{sup}}(s)=0$.
	\end{Thm}

	\subsection{The Carath$\acute{\text{e}}$odory singular dimension of a repeller}
	In~\cite{caopesinzhao2019}, Cao, Pesin and Zhao obtained that the Carath$\acute{\text{e}}$odory singular dimension of a repeller for a $C^{1+\gamma}$ map is exactly the unique root of the Bowen's equation involving the sub-additive topological pressure of the sub-additive singular valued potential. Take $\Phi_{f}(\alpha)=\{-\varphi^{\alpha}(\cdot,f^{n})\}_{n\geq 0}$ as the sub-additive singular valued potentials and denote by $P(f|_{\Lambda},\alpha)=P(f|_{\Lambda},\Phi_{f}(\alpha))$ the sub-additive topological pressure of the sub-additive singular-valued potential.
	\begin{Thm}\cite{caopesinzhao2019}\label{singular-valued dimension}
		Let $f: M\rightarrow M$ be a $C^{1+\gamma}$ expanding map with a repeller $\Lambda$. Assume that $f|_{\Lambda}$ is topologically transitive. Then $dim_{C,r} \Lambda=\alpha_{0}$ for all sufficiently small $r>0$, where $\alpha_{0}$ is the unique root of the Bowen's equation $P(f|_{\Lambda},\alpha)=0$.
	\end{Thm}
	\begin{Rem}
		As a result, we can use the notation $dim_{C}\Lambda$ to express the Carath$\acute{\text{e}}$odory singular dimension of the repeller $\Lambda$.
	\end{Rem}

	\subsection{Approximate of an hyperbolic measure by repellers} Results of the following type is usually referred to Katok~\cite{KH}.
	\begin{Thm}\label{approximate}
		Suppose $M$ is a $d-$dimensional Riemannian manifold. Let $f: M\rightarrow M$ be a $C^{1+\gamma}$ expanding map with a repeller $\Lambda$ and $\mu$ be an ergodic measure on $\Lambda$ with $h_{\mu}(f)>0$. Then for any $\varepsilon>0$, there exists an $f-$invariant compact subset $\Lambda_{\varepsilon}\subset\Lambda$ such that $h_{top}(f|_{\Lambda_{\varepsilon}})\geq h_{\mu}(f)-\varepsilon$ and for any $f-$invariant ergodic measure $\nu$ supported on $\Lambda_{\varepsilon}$, it holds that $|\lambda_{i}(\mu)-\lambda_{i}(\nu)|\leq\varepsilon$ for each $1\leq i\leq d$.
	\end{Thm}

	\section{Proof of the main results}\label{proof}
	In this section, we present the proof for our main results. %The proof is based on the lower bound for the repeller obtained by Cao, Pesin and Zhao in \cite{caopesinzhao2019}.
	\begin{Lem}\label{dim of measure}
		Suppose $f:M\rightarrow M$ is a $C^{1+\gamma}$ map on a $d-$dimensional Riemannian manifold and $\mu$ is an ergodic $f-$invariant measure which is absolutely continuous with the Lebesgue measure. Then $dim_{H} \mu=d$.
	\end{Lem}
	\begin{proof}
		For any measurable set $B$ with $\mu(B)=1$, by the absolute continuity of the measure $\mu$, we have $Leb(B)>0$. And hence $\dim_{H} B=d$. By the definition of $\dim_{H} \mu$, we get that $\dim_{H} \mu=d$.
	\end{proof}
	\begin{proof}[\textbf{Proof of Theorem~\ref{expanding}}]
		Since $f$ is expanding on $M$, there is an ergodic measure $\mu$ absolutely continuous with Lebesgue measure. Therefore, the Pesin's entropy formula ~\cite{BP} holds, that is
		\begin{align*}
			h_{\mu}(f)-\sum^{d}_{i=1} \lambda_{i}(\mu)=0.
		\end{align*}
		Hence, $$h_{\mu}(f)-\sum^{d-1}_{i=1} \lambda_{i}(\mu)=\lambda_{d}(\mu)>0.$$
		
		Fix $0<\varepsilon<\frac{\lambda_{d}(\mu)}{d}$ and $y\in M$, by Theorem~\ref{approximate}, there exists a repeller $\Lambda_{\varepsilon}$ satisfying $\Lambda_{\varepsilon}\cap \{y\}=\varnothing$ such that
		$$h_{\text{top}}(f|_{\Lambda_{\varepsilon}})\geq h_{\mu}(f)-\varepsilon$$
		and for any invariant ergodic measure $\nu$ supported on $\Lambda_{\varepsilon}$, we have
		$$|\lambda_{i}(\mu)-\lambda_{i}(\nu)|\leq \varepsilon\quad \text{for} \quad 1\leq i\leq d.$$
		Notice that $\Lambda_{\varepsilon}\cap \{y\}=\varnothing$ implies that $\Lambda_{\varepsilon}\subset E^{+}_{f}(y)$. Since $f$ is expanding on $\Lambda_{\varepsilon}$, there exists an invariant ergodic measure $\nu$ supported on $\Lambda_{\varepsilon}$ such that $$h_{\nu}(f|_{\Lambda_{\varepsilon}})=h_{\text{top}}(f|_{\Lambda_{\varepsilon}}).$$ Denote by $s^{\ast}$ the unique root of the equation $P_{\text{sup}}(s)=0$. By the definition of the super-additive pressure, we have
		\begin{equation}\label{entropy}
			h_{\nu}(f|_{\Lambda_{\varepsilon}})-\sum_{i=1}^{[s^{\ast}]}\lambda_{i}(\nu)-(s^{\ast}-[s^{\ast}])\lambda_{[s^{\ast}]+1}(\nu)\leq 0.
		\end{equation}
		By Theorem~\ref{approximate}, we have
		\begin{align*}
			\lambda_{d}(\mu)&=h_{\mu}(f)-\sum^{d-1}_{i=1} \lambda_{i}(\mu)\\
			&\leq h_{\nu}(f|_{\Lambda_{\varepsilon}})+\varepsilon-\sum_{i=1}^{d-1}\lambda_{i}(\nu)+(d-1)\varepsilon\\
			&<h_{\nu}(f|_{\Lambda_{\varepsilon}})-\sum_{i=1}^{d-1}\lambda_{i}(\nu)+\lambda_{d}(\mu).
		\end{align*}
		Thus, $$h_{\nu}(f|_{\Lambda_{\varepsilon}})-\sum_{i=1}^{d-1}\lambda_{i}(\nu)>0.$$ Combining this with (\ref{entropy}), we get that $s^{\ast}> d-1$. On the other hand, by Theorem~\ref{lower bound} and the monotonicity of the Hausdorff dimension, we have $$d\geq \dim_{H}E^{+}_{f}(y)\geq \dim_{H} \Lambda_{\varepsilon}\geq s^*.$$
		If $s^{\ast}=d$, there is nothing to prove. If $s^{\ast}<d$, then by Theorem~\ref{lower bound}, Theorem \ref{approximate} and (\ref{entropy}), we have
		\begin{align*}
			\dim_{H} \Lambda_{\varepsilon}&\geq s^{\ast}\\
			&\geq \frac{h_{\nu}(f|_{\Lambda_{\varepsilon}})-\sum_{i=1}^{[s^{\ast}]}\lambda_{i}(\nu)+[s^{\ast}]\lambda_{[s^{\ast}]+1}(\nu)}{\lambda_{[s^{\ast}]+1}(\nu)}\\
			&\geq \frac{h_{\mu}(f)-\varepsilon-\sum_{i=1}^{[s^{\ast}]}\lambda_{i}(\mu)-[s^{\ast}]\varepsilon+[s^{\ast}]\lambda_{[s^{\ast}]+1}(\mu)-[s^{\ast}]\varepsilon}{\lambda_{[s^{\ast}]+1}(\mu)+\varepsilon}\\
			&=\frac{([s^{\ast}]+1)\lambda_{[s^{\ast}]+1}(\mu)-(2[s^{\ast}]+1)\varepsilon}{\lambda_{[s^{\ast}]+1}(\mu)+\varepsilon}\\
			&=d-\frac{(3d-1)\varepsilon}{\lambda_{d}(\mu)+\varepsilon}.
		\end{align*}
		By the monotonicity of the Hausdorff dimension, we have $$d\geq \dim_{H}E^{+}_{f}(y)\geq \dim_{H} \Lambda_{\varepsilon}\geq d-\frac{(3d-1)\varepsilon}{\lambda_{d}(\mu)+\varepsilon}.$$ The arbitrariness of $\varepsilon>0$ yields that $\dim_{H}E^{+}_{f}(y)=d$. By Lemma \ref{dim of measure}, we finish the proof.

	\end{proof}
	%For the case that $f$ is not expanding on the whole manifold but admitting an repeller, we use the result in \cite{caopesinzhao2019} that the Carath$\acute{\text{e}}$odory singular dimension of the repeller is the unique root of the corresponding Bowen's equation to get our result.
	\begin{proof}[\textbf{Proof of Theorem~\ref{expanding2}}]
		Denote by $\alpha_{0}$ the unique root of $P(f|_{\Lambda},\alpha)=0$. By the variational principle of the sub-additive topological pressure and $f$ is expanding on $\Lambda$, there exists an ergodic measure $\mu^{\ast}$ such that
		\begin{align}\label{Pesin f1}
			h_{\mu^{\ast}}(f|_{\Lambda})-\sum^{d}_{i=d-[\alpha_{0}]+1}\lambda_{i}(\mu^{\ast})-(\alpha_{0}-[\alpha_{0}])\lambda_{d-[\alpha_{0}]}(\mu^{\ast})=0.
		\end{align}
		By Theorem~\ref{singular-valued dimension}, $\dim_{C} \Lambda=\alpha_{0}$.
		
		If $\alpha_{0}$ is an integer, then fix $0<\varepsilon<\frac{\lambda_{d-[\alpha_{0}]+1}(\mu^{\ast})}{[\alpha_{0}]}$, else fix $0<\varepsilon< \frac{(\alpha_{0}-[\alpha_{0}])\lambda_{d-[\alpha_{0}]}(\mu^{\ast})}{[\alpha_{0}]+1}$.
		For such an  $\varepsilon>0$ and any $y\in \Lambda$, by Theorem ~\ref{approximate}, there exists an $f-$invariant compact set $\Lambda_{\varepsilon}\subset \Lambda$ satisfying $\Lambda_{\varepsilon}\cap \{y\}=\varnothing$ such that $$h_{\text{top}}(f|_{\Lambda_{\varepsilon}})\geq h_{\mu^{\ast}}(f|_{\Lambda})-\varepsilon$$ and for any invariant measure $\nu$ supported on $\Lambda_{\varepsilon}$, we have $$|\lambda_{i}(\nu)-\lambda_{i}(\mu^{\ast})|\leq \varepsilon \quad \text{for} \quad 1\leq i\leq d.$$ Notice that $\Lambda_{\varepsilon}\subset E^{+}_{f|_{\Lambda}}(y)$. Since $f$ is expanding on $\Lambda_{\varepsilon}$, there exists an ergodic measure $\nu^{\ast}$ such that $$h_{\text{top}}(f|_{\Lambda_{\varepsilon}})=h_{\nu^{\ast}}(f|_{\Lambda_{\varepsilon}}).$$ Denote by $\alpha_{0}(\varepsilon)$ the unique root of $P(f|_{\Lambda_{\varepsilon}},\alpha)=0$.
		Then by the variational principle of the sub-additive topological pressure, we have
		\begin{equation}\label{entropy formula}
			h_{\nu^{\ast}}(f|_{\Lambda_{\varepsilon}})-\sum^{d}_{i=d-[\alpha_{0}(\varepsilon)]+1}\lambda_{i}(\nu^{\ast})-(\alpha_{0}(\varepsilon)-[\alpha_{0}(\varepsilon)])\lambda_{d-[\alpha_{0}(\varepsilon)]}(\nu^{\ast})\leq 0.
		\end{equation}
		If $\alpha_{0}$ is not an integer, then by Theorem \ref{approximate}, (\ref{Pesin f1}) and (\ref{entropy formula}), we have
		\begin{align*}
			(\alpha_{0}-[\alpha_{0}])\lambda_{d-[\alpha_{0}]}(\mu^{\ast})&=h_{\mu^{\ast}}(f|_{\Lambda})-\sum^{d}_{i=d-[\alpha_{0}]+1}\lambda_{i}(\mu^{\ast})\\
			&\leq h_{\nu^{\ast}}(f|_{\Lambda_{\varepsilon}})+\varepsilon-\sum^{d}_{i=d-[\alpha_{0}]+1}\lambda_{i}(\nu^{\ast})+[\alpha_{0}]\varepsilon\\
			&< h_{\nu^{\ast}}(f|_{\Lambda_{\varepsilon}})-\sum^{d}_{i=d-[\alpha_{0}]+1}\lambda_{i}(\nu^{\ast})+(\alpha_{0}-[\alpha_{0}])\lambda_{d-[\alpha_{0}]}(\mu^{\ast}).
		\end{align*}
		Combining this with (\ref{entropy formula}), we have \begin{align*}
			\sum^{d}_{i=d-[\alpha_{0}]+1}\lambda_{i}(\nu^{\ast})<\sum^{d}_{i=d-[\alpha_{0}(\varepsilon)]+1}\lambda_{i}(\nu^{\ast})+(\alpha_{0}(\varepsilon)-[\alpha_{0}(\varepsilon)])\lambda_{d-[\alpha_{0}(\varepsilon)]}(\nu^{\ast}).
		\end{align*}
		Therefore,
		$[\alpha_{0}]< \alpha_{0}(\varepsilon)$.
		On the other hand, by the monotonicity of the Carath$\acute{\text{e}}$odory singular dimension, we have $\dim_{C}\Lambda_{\varepsilon}\leq \dim_{C}\Lambda$, i.e. $\alpha_{0}(\varepsilon)\leq \alpha_{0}$. Hence, we have $[\alpha_{0}]=[\alpha_{0}(\varepsilon)]$.
		
		Therefore, by Theorem \ref{singular-valued dimension},  Theorem \ref{approximate}, (\ref{Pesin f1}) and (\ref{entropy formula}), we have
		\begin{align*}
			\text{dim}_{C}\Lambda_{\varepsilon}&=\alpha_{0}(\varepsilon)\geq\frac{h_{\nu^{\ast}}(f|_{\Lambda_{\varepsilon}})-\sum^{d}_{i=d-[\alpha_{0}(\varepsilon)]+1}\lambda_{i}(\nu^{\ast})+[\alpha_{0}(\varepsilon)]\lambda_{d-[\alpha_{0}(\varepsilon)]}(\nu^{\ast})}{\lambda_{d-[\alpha_{0}(\varepsilon)]}(\nu^{\ast})}\\
			&\geq\frac{h_{\mu^{\ast}}(f|_{\Lambda})-\varepsilon-\sum^{d}_{i=d-[\alpha_{0}]+1}\lambda_{i}(\mu^{\ast})-[\alpha_{0}]\varepsilon+[\alpha_{0}]\lambda_{d-[\alpha_{0}]}(\mu^{\ast})-[\alpha_{0}]\varepsilon}{\lambda_{d-[\alpha_{0}]}(\mu^{\ast})+\varepsilon}\\
			%&\geq \frac{h_{\mu^{\ast}}(f|_{\Lambda})-\sum^{d_{0}}_{i=d_{0}-[\alpha_{0}]+1}\lambda_{i}(\mu^{\ast})-(\alpha_{0}-[\alpha_{0}])\lambda_{d_{0}-[\alpha_{0}]}(\mu^{\ast})+\alpha_{0}\lambda_{d_{0}-[\alpha_{0}]}(\mu^{\ast})-(2[\alpha_{0}]+1)\varepsilon}{\lambda_{d_{0}-[\alpha_{0}]}(\mu^{\ast})+\varepsilon}\\
			&=\alpha_{0}-\frac{(\alpha_{0}+2[\alpha_{0}]+1)\varepsilon}{\lambda_{d-[\alpha_{0}]}(\mu^{\ast})+\varepsilon}.
		\end{align*}
		By the monotonicity of the Carath$\acute{\text{e}}$odory singular dimension, we have $$\dim_{C}\Lambda\geq \dim_{C} E^{+}_{f|_{\Lambda}}(y)\geq \dim_{C}\Lambda_{\varepsilon}\geq \alpha_{0}-\frac{(\alpha_{0}+2[\alpha_{0}]+1)\varepsilon}{\lambda_{d-[\alpha_{0}]}(\mu^{\ast})+\varepsilon}.$$ Let $\varepsilon$ tend to $0$. This yields the result.
		
		If $\alpha_{0}$ is  an integer, then by Theorem \ref{approximate}, (\ref{Pesin f1}) and (\ref{entropy formula}), we have
		\begin{align*}
			\lambda_{d-[\alpha_{0}]+1}(\mu^{\ast})&=h_{\mu^{\ast}}(f|_{\Lambda})-\sum^{d}_{i=d-[\alpha_{0}]+2}\lambda_{i}(\mu^{\ast})\\
			&\leq h_{\nu^{\ast}}(f|_{\Lambda_{\varepsilon}})+\varepsilon-\sum^{d}_{i=d-[\alpha_{0}]+2}\lambda_{i}(\nu^{\ast})+([\alpha_{0}]-1)\varepsilon\\
			&< h_{\nu^{\ast}}(f|_{\Lambda_{\varepsilon}})-\sum^{d}_{i=d-[\alpha_{0}]+2}\lambda_{i}(\nu^{\ast})+\lambda_{d-[\alpha_{0}]+1}(\mu^{\ast}).
		\end{align*}
		Combining this with ~(\ref{entropy formula}), we have 
		\begin{align*}
			\sum^{d}_{i=d-[\alpha_{0}]+2}\lambda_{i}(\nu^{\ast})<\sum^{d}_{i=d-[\alpha_{0}(\varepsilon)]+1}\lambda_{i}(\nu^{\ast})+(\alpha_{0}(\varepsilon)-[\alpha_{0}(\varepsilon)])\lambda_{d-[\alpha_{0}(\varepsilon)]}(\nu^{\ast}).
		\end{align*}
		
		Therefore, $[\alpha_{0}]-1< \alpha_{0}(\varepsilon)$. Since $\alpha_0(\varepsilon)\leq \alpha_0$ and $\alpha_{0}$ is  an integer, we have $\alpha_0-1<\alpha_0(\varepsilon)\leq \alpha_0$. Hence $\alpha=\alpha_{0}(\varepsilon)$ or $[\alpha_{0}]= [\alpha_{0}(\varepsilon)]+1$.
		
		If $\alpha_{0}=\alpha_{0}(\varepsilon)$, then the result follows from the monotonicity of the Carath$\acute{\text{e}}$odory singular dimension.
		
		If $[\alpha_{0}]=[\alpha_{0}(\varepsilon)]+1$, then by Theorem \ref{singular-valued dimension},  Theorem \ref{approximate}, (\ref{Pesin f1}) and (\ref{entropy formula}), we have
		\begin{align*}
			&\dim_{C}\Lambda_{\varepsilon}=\alpha_{0}(\varepsilon)\geq\frac{h_{\nu^{\ast}}(f|_{\Lambda_{\varepsilon}})-\sum^{d}_{i=d-[\alpha_{0}(\varepsilon)]+1}\lambda_{i}(\nu^{\ast})+[\alpha_{0}(\varepsilon)]\lambda_{d-[\alpha_{0}(\varepsilon)]}(\nu^{\ast})}{\lambda_{d-[\alpha_{0}(\varepsilon)]}(\nu^{\ast})}\\
			&\geq\frac{h_{\mu^{\ast}}(f|_{\Lambda})-\varepsilon-\sum^{d}_{i=d-[\alpha_{0}(\varepsilon)]+1}\lambda_{i}(\mu^{\ast})-[\alpha_{0}(\varepsilon)]\varepsilon+[\alpha_{0}(\varepsilon)]\lambda_{d-[\alpha_{0}(\varepsilon)]}(\mu^{\ast})-[\alpha_{0}(\varepsilon)]\varepsilon}{\lambda_{d-[\alpha_{0}(\varepsilon)]}(\mu^{\ast})+\varepsilon}\\
			&=\frac{h_{\mu^{\ast}}(f|_{\Lambda})-\sum^{d}_{i=d-[\alpha_{0}]+1}\lambda_{i}(\mu^{\ast})-(2[\alpha_{0}]-1)\varepsilon+[\alpha_{0}]\lambda_{d-[\alpha_{0}]+1}(\mu^{\ast})}{\lambda_{d-[\alpha_{0}]+1}(\mu^{\ast})+\varepsilon}\\ &=\alpha_{0}-\frac{(3[\alpha_{0}]-1)\varepsilon}{\lambda_{d-[\alpha_{0}]+1}(\mu^{\ast})+\varepsilon}.
		\end{align*}
		
		The monotonicity of the Carath$\acute{\text{e}}$odory singular dimension and the arbitrariness of  $\varepsilon>0$ imply the result.
	\end{proof}

\end{document}